\documentclass[twocolumn]{amsproc}
\usepackage{amsmath,amssymb}
\usepackage{abstract}
\usepackage{graphicx}

\newtheorem{theorem}{Theorem}[section]
\newtheorem{lemma}{Lemma}[section]
\newtheorem{corollary}{Corollary}[section]

\theoremstyle{remark}

\numberwithin{equation}{section}

\DeclareMathOperator{\RE}{Re}

\allowdisplaybreaks

\begin{document}

\title{Geometric Properties of Partial Sums of Univalent Functions }

\author[ V. Ravichandran]{V. Ravichandran}


\address{Department of Mathematics\\
University of Delhi\\ Delhi 110 007\\ India \and
School of Mathematical Sciences\\ Universiti Sains
Malaysia\\ 11800 USM, Penang, Malaysia}
\email{vravi@maths.du.ac.in}


\twocolumn[ \maketitle
\begin{onecolabstract}
The $n$th partial sum of an analytic function $f(z)=z+\sum_{k=2}^\infty a_k z^k$
is the polynomial $f_n(z):=z+\sum_{k=2}^n a_k z^k$. A survey of the univalence and other geometric properties
of the $n$th partial sum of univalent functions as well as other related   functions including those of   starlike, convex and close-to-convex functions are presented.
\end{onecolabstract}\vspace{-2.5cm}]

\section{Introduction}\noindent
For $r>0$, let $\mathbb{D}_r:=\{z\in\mathbb{C}:|z|<r\}$ be the open  disk of radius $r$ centered at $z=0$  and $\mathbb{D}:=\mathbb{D}_1$ be the open unit disk. An analytic function $f$ is \emph{univalent} in the unit disk $\mathbb{D}$ if it maps different points to different points. Denote the  class of all (normalized) univalent functions  of the form
\begin{equation}\label{fz}
 f(z)=z+\sum_{k=2}^\infty a_k z^k
\end{equation}
by $\mathcal{S}$. Denote by $\mathcal{A}$, the class of all analytic functions of the form \eqref{fz}.  The Koebe function $k$ defined by
\[ k(z)= \frac{z}{(1-z)^2} =z+\sum_{k=2}^\infty k a_k \quad (z\in \mathbb{D})\]
is univalent and it is also extremal for many problems in geometric function theory of univalent functions.
A domain $D$ is starlike with respect to a
point $a\in D$ if every line segment joining the point $a$ to any
other point in $D$ lies completely inside $D$. A domain starlike
with respect to the origin is  simply called starlike. A domain $D$
is convex if every line segment joining any two points in $D$ lies
completely inside $D$; in other words, the domain $D$ is convex if
and only if it is starlike with respect to every point in $D$. A
function $f\in \mathcal{S}$ is \emph{starlike} if $f(\mathbb{D})$ is
starlike (with respect to the origin) while it is \emph{convex} if
$f(\mathbb{D})$ is convex. The classes of all starlike and convex
functions are respectively denoted by $\mathcal{S}^*$ and
$\mathcal{C}$. Analytically, these classes are characterized by the
equivalence
\[f\in\mathcal{S}^*\Leftrightarrow \RE \left( \frac{zf'(z)}{f(z)} \right) > 0, \] and
\[ f\in\mathcal{C} \Leftrightarrow \RE \left(1+ \frac{zf''(z)}{f'(z)}
\right) > 0 .  \] More generally, for $ 0\leq \alpha<1$, let
$\mathcal{S}^* (\alpha)$ and $\mathcal{C}(\alpha)$ be the subclasses
of $\mathcal{S}$ consisting of respectively starlike functions of
order $\alpha$, and convex functions of order $\alpha$. These
classes are defined analytically by the equivalence
\[f\in\mathcal{S}^*(\alpha)\Leftrightarrow \RE \left( \frac{zf'(z)}{f(z)} \right)
> \alpha, \]
  and \[ f\in\mathcal{C}(\alpha) \Leftrightarrow \RE \left(1+
\frac{zf''(z)}{f'(z)} \right) >\alpha .  \] Another related
class is the class of close-to-convex functions. A function $f\in
\mathcal{A}$ satisfying the condition
\[\RE\left( \frac{f'(z)}{g'(z)}\right) > \alpha  \quad ( 0\leq \alpha <1 ) \]
for some (not necessarily normalized) convex univalent function $g$,
is called \emph{close-to-convex   of order $\alpha$}. The class of all such
functions is denoted by $\mathcal{K}(\alpha)$. Close-to-convex
functions of order $0$ are simply called close-to-convex functions.
Using the fact that a function $f\in\mathcal{A}$ with
\[\RE(f'(z))>0\] is in $\mathcal{S}$,  close-to-convex functions can
be shown to be univalent.

The $n$th \emph{partial sum} (or $n$th \emph{section}) of the function $f$, denoted by $f_n$, is the polynomial defined by
\[ f_n(z):=z+\sum_{k=2}^n a_k z^k .\]
The second partial sum $f_2$ of the Koebe function $k$ is given by
\[ f_2(z)=z+2z^2 \quad (z\in\mathbb{D}). \]
It is easy to check directly (or by using the fact that $|f_2'(z)-1|<1$ for  $|z|<1/4$) that this function $f_2$ is univalent in the disk $\mathbb{D}_{1/4}$  but, as $f_2'(-1/4)=0$, not in any larger disk. This simple example shows that the partial sums of univalent functions need not be univalent in $\mathbb{D}$.

The second partial sum of the function $f(z)=z+\sum_{k=2}^\infty  a_k z^k $ is the function $f_2(z)=z+a_2z^2$. If $a_2=0$, then $f_2(z)=z$ and its properties are clear. Assume that $a_2\neq 0$. Then the function $f_2$ satisfies, for $|z|\leq r$, the inequality
\[ \RE\left(\frac{zf_2'(z)}{f_2(z)}\right) = \RE\left(1+\frac{a_2z}{1+a_2z}\right) \geq 1- \frac{|a_2|r}{1-|a_2|r}
>0 \]
provided $r< 1/(2|a_2|)$. Thus the radius of starlikeness of $f_2$ is $1/(2|a_2|)$. Since $f_2$ is convex in $|z|<r$ if and only if $zf_2'$ is starlike in $|z|<r$, it follows that the radius of convexity of $f_2$ is $1/(4|a_2|)$. If $f$ is univalent or starlike univalent, then $|a_2|\leq 2$ and therefore the radius of univalence of $f_2$ is 1/4 and the radius of convexity of $f_2$  is 1/8. (See Fig.~\ref{fig:kobe11} for the image of $\mathbb{D}_{1/4}$ and $\mathbb{D}_{1/8}$ under the function $z+2z^2$.)  For a convex function $f$ as well as  for functions $f$ whose derivative has positive real part in $\mathbb{D}$, $|a_2|\leq 1 $ and so the radius of univalence for  the second partial sum $f_2$ of these functions  is 1/2 and the radius of convexity is 1/4. In \cite{owa}, the starlikeness and convexity of the  initial partial sums of the Koebe function $k(z)=z/(1-z)^2$ and the function $l(z)=z/(1-z)$ are investigated.

\begin{figure}[hbp!]\label{fig:kobe11}
\includegraphics[scale =.5, bb = 0  0  360 404]{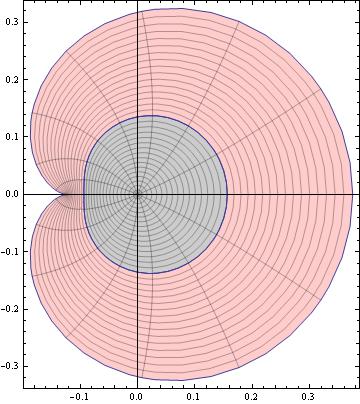}
\caption{Images of $\mathbb{D}_{1/8}$ and  $\mathbb{D}_{1/4}$ under the mapping $w=z+2z^2$. }
\end{figure}

It is therefore of interest to determine the largest disk $\mathbb{D}_\rho$ in which the partial sums of the univalent functions are univalent.   Szeg\"o also wrote a survey \cite{szego36} on partial sums  in 1936. In the present survey, the various results on the partial sums  of functions belonging to the subclasses of univalent functions  are given. However, Ces\`aro means and other polynomials
approximation of univalent functions are not considered in this survey.

\section{Partial sums of univalent functions}

 The second partial sum of the Koebe function   indicates that the   partial sums of univalent functions cannot be univalent in a disk of radius larger than 1/4. Indeed, by making use of Koebe's distortion theorem and L\"owner's theory of univalent functions, Szeg\"o \cite{szego28} in 1928 proved the following theorem.

\begin{theorem}[Szeg\"o  Theorem]The partial sums of univalent functions $f\in\mathcal{S}$ are univalent in the disk  $\mathbb{D}_{1/4}$ and the number $1/4$ cannot be replaced by a larger one.
\end{theorem}

Using an inequality of Goluzin, Jenkins \cite{jenkins51} (as well as Ilieff \cite{ilieff51}, see Duren \cite[\S8.2, pp. 241--246]{durenbook}) found  a simple proof of this result and  also shown that the partial sums of odd univalent functions are univalent in $\mathbb{D}_{1/\sqrt{3}}$. The number $1/\sqrt{3}$ is  shown to be the radius of starlikeness of the partial sums of the odd univalent functions by He and Pan \cite{hepan}.
Iliev \cite{iliev} investigated the radius of univalence for the partial sums $\sigma^{(k)}_n(z)=z+c^{(k)}_1z^{k+1}+\cdots+c^{(k)}_nz^{nk+1}$, $n=1,2,\dotsc$, of univalent function of the form $f_k(z)=z+c^{(k)}_1z^{k+1}+\cdots$. For example, it is shown that $\sigma^{(2)}_n$ is univalent in $\vert z\vert <1/\sqrt{3}$, and $\sigma^{(3)}_n$ is univalent in $\vert z\vert <\root 3\of{3}/2$, for all $n=1,2,\dotsc$.  He has also shown that $\sigma^{(1)}_n(z)$ is univalent in $\vert z\vert <1-4(\ln n)/n$ for $n\geq 15$. Radii of  univalence are also determined for $\sigma^{(2)}_n$ and $\sigma^{(3)}_n$, as functions of $n$, and for $\sigma^{(k)}_1$ as a function of $k$.

Szeg\"o's theorem that the partial sums of univalent functions are univalent in $\mathbb{D}_{1/4}$ was strengthened  to starlikeness by Hu and Pan \cite{Hu84}. Ye \cite{ye} has shown that  the partial sums of univalent functions are convex in $\mathbb{D}_{1/8}$ and that the number 1/8  is sharp.  Ye \cite{ye} has proved the following result.

\begin{theorem}Let $f\in \mathcal{S}$ and \[f^{1/k}(z^k)=\sum_{\nu=0}^\infty b_\nu^{(k)} z^{\nu k+1},\quad (k=2,3,\dotsc,\ b_0^{(k)}=1.)\]  Then $ \sum_{\nu=0}^n b_\nu^{(k)} z^{\nu k+1}$ are convex in $\mathbb{D}_{\sqrt[k]{k/(2(k+1)^2)}}$. The radii of convexity  are sharp.
\end{theorem}

Ruscheweyh gave an extension of Szeg\"o's theorem that the $n$th partial sums $f_n$ are starlike in $\mathbb{D}_{1/4}$ for functions belonging not only to  $\mathcal{S}$ but also to the closed convex hull of $\mathcal{S}$.

Let $\mathcal{F}=\operatorname{clco}\{ \sum_{k=1}^n x^{k-1}z^k: |x|\leq 1\}$ where $\operatorname{clco}$ stands for the closed convex hull.
Convolution of two analytic functions $f(z)=z+\sum_{k=2}^\infty a_k z^k$ and $g(z)=z+\sum_{k=2}^\infty b_k z^k$ is the function $f*g$ defined by
\[ (f*g)(z):= z+\sum_{k=2}^\infty a_k b_k z^k .\] Ruscheweyh \cite{Rush88} proved the following theorem.

\begin{theorem}\label{th:Rush}
If $f\in\operatorname{clco} \mathcal{S}$ and $g\in\mathcal{F}$, then $f*g$ is starlike in $\mathbb{D}_{1/4}$. The constant 1/4 is best possible.
\end{theorem}

In particular, for $g(z)=z+z^2+\cdots+z^n$, Theorem~\ref{th:Rush} reduces to the following result.
\begin{corollary}If $f$ belongs to $\operatorname{clco} \mathcal{S}$ or, in particular, to the class of the normalized  typically real functions, then the $n$th partial sum $f_n$ is starlike in $\mathbb{D}_{1/4}$. The constant 1/4 is best possible.
\end{corollary}

The class $\mathcal{F}$ contains the following two subsets:
\[\mathcal{R}_{1/2}:=\left\{ f\in\mathcal{A}: \RE(f(z)/z)>1/2, z\in \mathbb{D}\right\}\subset \mathcal{F} \]
and
\[\mathcal{D}:=\left\{  \sum_{k=1}^n a_k z^k \in \mathcal{A}: 0\leq a_{k+1}\leq a_k \right\}\subset \mathcal{F}. \]
Since the class $\mathcal{C}$ of convex functions is a subset of $\mathcal{R}_{1/2}$, it is clear that $\mathcal{S}^*\subset \mathcal{C}\subset \mathcal{F}$.  For $g(z)=z/(1-z)^2 \in \mathcal{S}^*$,  Theorem~\ref{th:Rush} reduces to the following:
\begin{corollary}If $f$ belongs to $ \mathcal{F}$, then the function $f$ and, in particular, the $n$th partial sum $f_n$,  is convex in $\mathbb{D}_{1/4}$. The constant 1/4 is best possible.
\end{corollary}

We remark that  Suffridge \cite{suff} has shown that the partial sums of the function $e^{1+z}$ are all convex. More generally, Ruscheweyh and Salinas \cite{russal}
have shown that the functions of the form $ \sum_{k=0}^\infty a_k(1+z)^k/k!$, $a_0\geq a_1\geq\dots\geq0$  are either constant or convex univalent in the unit disk $\mathbb{D}$.
Let  $F(z)=z+\sum_1^\infty a_kz^{-k}$ be analytic   $|z|>1$.   Reade \cite{reade} obtained the radius of univalence for the partial sums  $F_n(z)=z+\sum_1^na_kz^{-k}$ when $F$ is  univalent or when  $\RE F'(z)>0$ in $|z|>1$.

\section{Partial sums of starlike functions}

Szeg\"o \cite{szego28} showed that the partial sums of starlike (respectively convex) functions are starlike (respectively convex) in the disk  $\mathbb{D}_{1/4}$ and the number $1/4$ cannot be replaced by a larger one. If $n$ is fixed, then the radius of starlikeness of $f_n$ can be shown to depend on $n$. Motivated by a result of Von Victor Levin that the $n$th partial sum of univalent functions is univalent in
$D_\rho$ where $\rho=1-6(\ln n)/n$ for $n\geq 17$, Robertson \cite{robert36} determined $R_n$ such that  the $n$th partial sum $f_n$ to have certain property $P$ in $\mathbb{D}_{R_n}$ when the function $f$ has the property $P$ in $\mathbb{D}$. He considered the function  has one of the following properties: $f$ is  starlike, $f/z$ has positive real part, $f$ is convex, $f$ is typically-real or $f$ is convex in the direction of the imaginary axis and is real on the real axis. An error in the expression for $R_n$ was later corrected in his paper \cite{robert41} where he has extended his results to multivalent starlike functions.

The radius of starlikeness of the $n$th partial sum of starlike function is given in the following theorem.

\begin{theorem}\cite {robert41} {\rm (see \cite[Theorem 2, p.\ 1193]{silver88})}
If $f(z)=z+\sum_{k=2}^\infty a_k z^k$ is either starlike, or    convex, or typically-real, or convex in the direction of imaginary axis, then there is $n_0$ such that, for $n\geq n_0$,   the partial sum
$f_n(z):=z+\sum_{k=2}^n a_k z^k$ has the same property  in $\mathbb{D}_\rho$ where $\rho \geq 1-3 \log n/n$.
\end{theorem}

An analytic function $f(z)=z^p+\sum_{k=1}^\infty a_{p+k}z^{p+k}$ is $p$-valently starlike \cite[p.\ 830]{robert41} if $f$ assumes no value more than $p$ times, at least one value $p$ times and
\[\RE\left(\frac{zf'(z)}{f(z)} \right)>0 \quad (z\in \mathbb{D}).\] For $p$-valently starlike functions, Robertson \cite[Theorem A, p. 830]{robert41} proved that the radius of $p$-valently starlikeness of the $n$th partial sum $f_n(z)= z^p+\sum_{k=1}^n a_{p+k}z^{p+k}$ is  at least $1-(2p+2)\log n/n$. Ruscheweyh \cite{rusc-jipam} has given a simple proof that the partial sums $f_n$ of $p$-valently starlike (or close-to-convex) function is $p$-valently starlike (or respectively close-to-convex) in $|z|<1/(2p+2)$.

Kobori \cite{kobori} proved the following theorem and Ogawa \cite{ogawa2} gave another proof of this result.
\begin{theorem}
If $f$ is a starlike function, then every partial sum $f_n$ of $f$ is convex in $|z|<1/8$ and the number 1/8 cannot be increased.
\end{theorem}

In view of the above theorem,  the $n$th partial sum of Koebe function $z/(1-z)^2$ is convex in $|z|<1/8$. A verification of this fact directly can be used to give another proof of this theorem by using the fact \cite{Rush73}  that the convolution of  two convex function is again convex. It is also known \cite{Rush73} that $\RE(f(z)/f_n(z))>1/2$ for a function $f$ starlike of order 1/2. This result was extended  by Singh and Paul \cite{singhpaul} in the following theorem.
\begin{theorem}If $f\in \mathcal{S}^*(1/2)$, then \[ \RE\left(\lambda \frac{zf'(z)}{f(z)}+\mu \frac{f_n(z)}{f(z)}\right)>0\quad (z\in \mathbb{D})\] provided that $\lambda$ and $\mu$ are both nonnegative with at least one of them nonzero or provided that $\mu$ is a complex number with $|\lambda|> 4|\mu|$.
The result is sharp in the sense that the ranges of $\lambda$ and $\mu$ cannot be increased.
\end{theorem}


\section{Partial sums of convex functions}

For a convex function $f\in\mathcal{C}$, it is well-known that $\RE(f(z)/z)>1/2$. Extending this result, Sheil-Small \cite{sheilsmall0} proved the following theorem.

\begin{theorem}If  $f\in\mathcal{C}$, then the $n$th partial sum $f_n$ of $f$ satisfies
\[\left|1-\frac{f_n(z)}{f(z)}\right|\leq |z|^n<1\quad (z\in\mathbb{D}, n\geq 1) \]and hence
\begin{equation} \label{convex:eq1}
\RE \frac{f(z)}{f_n(z)} > \frac{1}{2} \quad (z\in\mathbb{D},n\geq 1).
\end{equation}
\end{theorem}
As a consequence of this theorem, he has shown that the function $Q_n$ given by $Q_n(z)=\int_0^z (f_n(\eta)/\eta )d\eta $ is a close-to-convex  univalent function. In fact,
the inequality \eqref{convex:eq1}  holds for $f\in \mathcal{S}^*(1/2)$ as shown in \cite{Rush73}. The inequality \eqref{convex:eq1} also holds for odd starlike functions as well as for functions whose derivative has real part greater than 1/2 \cite{ramsinghpuri}.

Recall once again that
\cite{szego28} the partial sums of convex  functions are  convex  in the disk  $\mathbb{D}_{1/4}$ and the number $1/4$ cannot be replaced by a larger one. A different proof of this result can be given by using results about convolutions of convex functions.
In 1973, Ruscheweyh and Shiel-Small proved the Polya-Schoenberg conjecture (of 1958)  that the convolution of two convex univalent functions is again convex univalent. Using this result,  Goodman and Schoenberg gave another proof  of the following result of Szeg\"o \cite{szego28}.

\begin{theorem}\label{psconvex}
If $f$ is convex function, then every partial sum $f_n$ of $f$ is convex in $|z|<1/4$.
\end{theorem}

\begin{proof}
The convex function $l(z)=z/(1-z)$ is extremal for many problems associated with the class of convex functions. By Szeg\"o's result the partial sum $l_n(z)=z+z^2+\cdots+z^n$ of $l$ is convex in $\mathbb{D}_{1/4}$ and therefore $4l_n(z/4)$ is a convex univalent function. If $f$ is also convex, then its convolution with the convex function $4l_n(z/4)$ is convex and so $4(f*l_n)(z/4)= f(z)*(4l_n(z/4))$ is convex. Therefore, the partial sum $f_n$ of the convex function $f$, as  $f*l_n=f_n$, is convex in $\mathbb{D}_{1/4}$. In view of this discussion, another proof of Szeg\"o result comes if we can directly show that $l_n(z)$ is convex in $\mathbb{D}_{1/4}$.   This will be done below.

A calculation shows that
\begin{align*}
 1+ \frac{zl_n''(z)}{l_n'(z)} &= \frac{n(n+1)z^n(z-1)}{1-(n+1)z^n+nz^{n+1}}+\frac{1+z}{1-z}\\
 & =\frac{N(z)}{D(z)}+M(z)\end{align*}
where $N(z)=n(n+1)z^n(z-1)$, $D(z)=1-(n+1)z^n+nz^{n+1}$ and $M(z)= (1+z)/(1-z)$. The bilinear transformation $w=M(z)$ maps $|z|< 1/4$ onto  the disk  $|w-17/15|<8/15$ and hence
\[ \RE M(z) >3/5. \]It is therefore enough to show that
\[ \left|\frac{N(z)}{D(z)}\right|<\frac{3}{5}\] as this inequality together with $\RE M(z) >3/5 $ yield
\[\RE\left(\frac{N(z)}{D(z)}+M(z)\right) \geq \RE M(z)- \left|\frac{N(z)}{D(z)}\right| > \frac{3}{5}-\frac{3}{5}=0.
\] Now, for $|z|<1/4$, we have
 \begin{align*} |N(z)|&<\frac{5n(n+1)}{4^{1+n}},  \\
|D(z)-1| & = |(n+1)z^n-nz^{n+1}| \\
& < \frac{1}{4^n}(n+1+n/4)\\
&  = \frac{5n+4}{4^{n+1}} <1
\end{align*}
and so
\[ |D(z)| \geq 1-|D(z)-1|> 1-  \frac{5n+1}{4^{n+1}}.  \]  Therefore, it follows that
\[\left|\frac{N(z)}{D(z)}\right|<\frac{3}{5} \] holds if
\[  \frac{5n(n+1)}{4^{n+1} } <\frac{3}{5}\left( 1-  \frac{5n+4}{4^{n+1}}\right)   \] or
equivalently
\[ \frac{25}{12} \leq \frac{4^n}{n(n+1)}-\frac{1}{n} -\frac{1}{4(n+1)}. \]
The last inequality becomes an equality for $n=2$ and the right hand side expression is an increasing function of
$n$.
\end{proof}


Let $P_{\alpha,n}$ denote the class of functions $p(z)=1+c_nz^n+\cdots\ (n\geq 1)$ analytic and satisfying the condition $\text{Re}\,p(z)>\alpha\ (0\leq\alpha<1)$ for $z\in\mathbb{D}$. Bernardi \cite{bernardi74} proved   that the sharp inequality \[ \frac{|zp'(z)|}{\RE(p(z)-\alpha)} \leq \frac{2nr^n}{1-r^{2n}} \] holds for $p(z)\in P_{\alpha,n}$, $|z|=r<1$, and $n=1,2,3,\dotsc$.  He has also shown that, for any complex $\mu$, $\text{Re}\,\mu=\beta>0$, \[ \left|\frac{zp'(z)}{p(z)-\alpha+(1-\alpha)\mu}\right| \leq \frac{2nr^n}{(1-r^n)(1+\beta+(1-\beta)r^n)}. \]    For a convex function $f$, he deduced  the sharp inequality
\[ \left|\frac{zf'(z)}{f(z)}-\frac{zf_n'(z)}{f_n(z)}\right|\leq \frac{nr^n}{1-r^n}.\]
Making use of this inequality, he proved the following theorem.

\begin{theorem}\cite[Theorem 4, p.\ 117]{bernardi74}\label{convex:szgo}
If $f$ is convex, then the $n$th partial sum $f_n$ is starlike in $|z|<r_n$ where $r_n$ is the positive root of the equation $1-(n+1)r^n-nr^{n+1}=0$. The result is sharp for each even $n$ for $f(z)=z/(1-z)$.
\end{theorem}

The above theorem with a weaker conclusion that $f_n$ is univalent was obtained earlier by Ruscheweyh \cite{Rush72}. Singh \cite{ramsingh88} proved that  the $n$th partial sum $f_n$ of a  starlike function of order 1/2 is starlike in $|z|<r_n$ where $r_n$ is given in Theorem~\ref{convex:szgo}. He has also shown that the conclusion of Theorem~\ref{convex:szgo} can be strengthened to convexity if one assumes that $f$ is a convex function of order 1/2. In addition, for  a convex function $f$ of order 1/2, he has shown that
$\RE (f_n(z)/z)>1/2$ and 1/2 is sharp. It is also known that all the partial sums of a convex function $f$ of order 1/2 are close-to-convex with respect to $f$ itself and that there are  convex functions of order $\alpha<1/2$ whose partial sums are not univalent \cite{robert-boundary}.  Singh and Puri \cite{ramsinghpuri} have however shown that  each of the partial sums of an odd convex function $f$ is close-to-convex with respect to $f$.

Silverman \cite{silver88} also proved Theorem~\ref{convex:szgo} by finding the  radius of starlikeness (of order $\alpha$) of the $n$th partial sums of the function $z/(1-z)$. The result then follows from the fact that the classes of convex and starlike functions are closed under convolution with convex functions.

\begin{lemma} \label{c5lem1}
The function $g_n(z)=\frac{z(1-z^n)}{1-z}$ is starlike of order $\alpha$
in $|z| <r_n$ where $r_n$ is the smallest positive root of the equation
$$1-\alpha-\alpha r +(\alpha-1-n)r^n+(\alpha-n)r^{n+1} =0. $$
The result is sharp for even $n$.
\end{lemma}

\begin{proof} The bilinear transformations   $w=1/(1-z)$ maps the circular region  $|z| \leq r $   onto the circle
$$\left| \frac{1}{1-z} -\frac{1}{1-r^2}\right|\leq \frac{r}{1-r^2}. $$
Similarly
$$\left| \frac{z}{1-z} -\frac{r^2}{1-r^2}\right|\leq \frac{r}{1-r^2}. $$
Since
$$\frac{zg_n'(z)}{g_n(z)} = \frac{1}{1-z}-\frac{nz^n}{1-z^n}, $$
it follows  that, for $|z| \leq r < 1$,
$$\left|\frac{zg_n'(z)}{g_n(z)} - \frac{1}{1-r^2}+\frac{nr^{2n}}{1-r^{2n}} \right|
\leq \frac{r}{1-r^2}+\frac{nr^n}{1-r^{2n}}.  $$
The above inequality shows that
$${\rm Re\, } \frac{zg_n'(z)}{g_n(z)} \geq \frac{1}{1+r} -\frac{nr^n}{1-r^n}
\geq \alpha $$
provided
$$1-\alpha-\alpha r +(\alpha-1-n)r^n+(\alpha-n)r^{n+1} \geq 0. $$
The sharpness follows easily. \end{proof}

\begin{theorem}\cite[Theorem 1, p.\ 1192]{silver88}
If $f$ is convex, then the $n$th partial sum $f_n$ is starlike in $|z|<(1/(2n))^{1/n}$ for all $n$. In particular,
$f_n$ is starlike in $|z|<1/2$ and the radius 1/2 is sharp.
\end{theorem}

\begin{proof}In view of the previous lemma, it is enough to show that
\[ 1-(n+1)r^n-nr^{n+1}\geq 0  \]
for $0\leq r \leq (1/(2n))^{1/n}$. For $0\leq r \leq (1/(2n))^{1/n}$, the above inequality is equivalent to
\[   \frac{1}{n}+\frac{1}{(2n)^{1/n}}\leq 1, \] which holds for $n\geq 2$. 

The second result follows as $1/(2n)^{1/n}$ is an increasing function of $n$ and from the fact that, for  $g_2(z)=z+z^2$, $g_2'(-1/2)=0$.
\end{proof}

Silverman \cite[Corollary 2, pp. 1192]{silver88} also proved that the $n$th partial sum $f_n$ of a convex function $f$ is starlike in $|z|<\sqrt{23/71}$  for $n\geq 3$ and the radius $\sqrt{23/71} $ is sharp. For a convex function $f$, its $n$th partial sum $f_n$ is  shown to be  starlike of order $\alpha$ in $|z|< (1-\alpha)/(2-\alpha)$, convex   of order $\alpha$ in $|z|< (1-\alpha)/(2(2-\alpha))$ and the radii are sharp.

A function $f\in \mathcal{S}$ is \emph{uniformly
convex}, written $f\in \mathcal{UCV}$, if $f$ maps every circular arc $\gamma$ contained in
$\mathbb{D}$ with center $\zeta\in \mathbb{D}$  onto a convex arc. The class $\mathcal{S}_P$ of \emph{parabolic starlike functions}  consists of
functions $f \in \mathcal{A} $ satisfying
\[   \label{eqsp}
 \RE     \left(    \frac{zf'(z)}{f(z)}     \right)     >
\left| \frac{zf'(z)}{f(z)} - 1\right|, \quad  z\in \mathbb{D} . \]
In other words, the class $\mathcal{S}_P$ consists of functions $f=zF'$ where $
F \in  \mathcal{UCV}$. A survey of these classes can be found in \cite{ucvsurvey}.

\begin{lemma} \label{c5lem2}
The function $g_n(z)=\frac{z(1-z^n)}{1-z}$ is in $\mathcal{S}_P$
for  $|z| <r_n$ where $r_n$ is the smallest positive root of the equation
$$1- r = (1+2n)r^n+(2n-1)r^{n+1} . $$
The result is sharp for even $n$.
\end{lemma}

\begin{proof}  By Lemma \ref{c5lem1}, the function $g_n$ is starlike of order
1/2 in $|z| < r_n$ where $r_n$ is as given in the Lemma
\ref{c5lem2}. From the proof of Lemma \ref{c5lem1}, it
follows that, for $|z|=r$, the values of $zg_n'(z)/g_n(z)$ is in
the disk with diametric end points at
$$x_1=\frac{1}{1+r}-\frac{nr^n}{1-r^n} \mbox{ and }
x_2=\frac{1}{1-r}+\frac{nr^n}{1+r^n}. $$ For $r=r_n$, one has
$x_1=1/2$ and the disk is completely inside the parabolic region
${\rm Re\, } w >|w-1|$.
\end{proof}

\begin{theorem}  \label{c5th4}
If $f(z)$ is convex of order 1/2, then the partial sums $f_n$ are
uniformly convex for $|z|<r_n$ where $r_n $ is the smallest positive root of
$$1-r= (2n+1)r^n+(2n-1)r^{n+1}. $$
\end{theorem}

\begin{proof} For the function $f(z)$ convex of order 1/2, the
function
$zf'(z)$ is  starlike of order 1/2. Since $\int_0^z g_n(t)/tdt$
is uniformly convex in $|z| <r_n$, $$f_n(z;f) = f(z)*g_n(z)=
zf'(z)*(\int_0^z g_n(t)/tdt) $$ is uniformly convex in $|z|<r_n$.
\end{proof}

\section{Functions whose derivative is bounded or has positive real part}
\begin{theorem}\cite{noshiro32} \label{th:noshiro32}
If the  analytic  function $f$ given by \eqref{fz} satisfies the inequality $|f'(z)|\leq M$, $M>1$, then the radius of starlikeness of $f_n$ is  $1/M$.
\end{theorem}

\begin{proof} The function $f$ is starlike in $\mathbb{D}_r$ if \[\sum_{k=2}^\infty k |a_k|r^{k-1}\leq 1.\] This sufficient condition is now well-known  (Alexandar II) and it was  also proved by Noshiro \cite{noshiro32}.  The Parseval-–Gutzmer formula for a function $f(z)=\sum_{k=0}^\infty  a_k z^k $ analytic in $\overline{\mathbb{D}}_r$ is
\[\int^{2\pi}_0 |f(re^{i\vartheta}) |^2 \, \mathrm{d}\vartheta = 2\pi \sum^\infty_{k = 0} |a_k|^2r^{2k}.\]
Using this formula for $f'$ and noting that $|f'(z)|\leq M$, it follows that
\[1+\sum_{k=2}^\infty k^2|a_k|^2=\lim_{r\rightarrow 1}  \frac{1}{2\pi} \int^{2\pi}_0 |f'(re^{i\vartheta}) |^2 \, \mathrm{d}\vartheta \leq M^2.  \]
Now, by using the Cauchy-–Schwarz inequality, it readily follows that, for $r<1/M$,
\begin{align*}
\sum_{k=2}^\infty k |a_k|r^{k-1}& \leq  \sqrt{\sum_{k=2}^\infty k^2 |a_k|^2}\sqrt{ \sum_{k=2}^\infty r^{2k-2}}\\
& \leq \sqrt{M^2-1}\sqrt{\frac{r^2}{1-r^2}}\\
& < 1.
\end{align*}
The sharpness follows from the function $f_0$ given by
\begin{align*}
 f_0(z)& = M\int_0^z \frac{1-Mz}{M-z}dz\\
 &=M\left(Mz+(M^2-1)\log\left(1-\frac{z}{M}\right)\right);\end{align*}
its derivative vanishes at $z=1/M$.
\end{proof}
Nabetani \cite{nabetani34} noted that Theorem~\ref{th:noshiro32} holds even if the inequality $|f'(z)|\leq M$ is replaced by the inequality
\[ \left(\frac{1}{2\pi} \int^{2\pi}_0 |f'(re^{i\vartheta}) |^2 \, \mathrm{d}\vartheta \right)^{1/2}\leq M. \]
He has shown that the radii of starlikeness and convexity  of functions $f$ satisfying the inequality
\[ \left(\frac{1}{2\pi} \int^{2\pi}_0 |f(re^{i\vartheta}) |^2 \, \mathrm{d}\vartheta \right)^{1/2}\leq M\]
are respectively  the positive root of the equations \[ (M^2-1)R^2=(1-R^2)^3\] and
\[ (M^2-1)(1+11R^2+11R^4+R^6)=M^2(1-R^2)^5.\]

For functions whose derivative has  positive real part, MacGregor \cite{mac} proved the following result.
\begin{theorem}\label{macth} If the  analytic  function $f$ given by \eqref{fz} satisfies the inequality $\RE f'(z)>0$, then  $f_n$ is univalent in $|z|<1/2$.
\end{theorem}
\begin{proof}Since $\RE f'(z)>0$, $|a_k|\leq 2/k$, $(k\geq 2)$, and so, with $|z|=r$,
\[ |f'(z)-f_n'(z)|\leq \sum_{k=n+1}^\infty |k a_k z^{k-1}|\leq \sum_{k=n+1}^\infty 2r^{k-1}=\frac{2r^{n}}{1-r}.\] This together with the estimate $\RE f'(z)> (1-r)/(1+r)$ shows that
\[ \RE f'_n(z)\geq \frac{1-r}{1+r}- \frac{2r^{n}}{1-r}.\]
The result follows from this for $n\geq 4$. For $n=2,3$, a different analysis is needed, see \cite{mac}. Compare Theorem~\ref{th31}.
\end{proof}

\begin{theorem}\label{th31}\cite{robert36}
If $p(z)=1+c_1z+c_2z^2+\cdots$ is analytic and has positive real part in $\mathbb{D}$, then, for $n\geq 2$, $p_n(z)=1+c_1z+c_2z^2+\cdots+c_nz^n$ has positive real part in $\mathbb{D}_\rho$ where
$\rho$ is the root $R_n\geq 1-2 \log n/n$ in (0,1) of the equation
\[ (1-r)^2=2r^{n+1}(1+r). \]
\end{theorem}%

Singh \cite{ramsingh} investigated the radius of convexity for functions whose derivative has  positive real part and proved the following result.

\begin{theorem}If the  analytic  function $f$ given by \eqref{fz} satisfies the inequality $\RE f'(z)>0$, then  $f_n$ is convex in $|z|<1/4$. The number 1/4 cannot be replaced by a greater one.
\end{theorem}

Extending Theorem~\ref{macth} of MacGregor, Silverman \cite{silvertam} has shown that, whenever $\RE f'(z)>0$,  $f_n$ is univalent in $\{z\colon\ |z|<r_n\}$, where $r_n$ is the smallest positive root of the equation $1-r-2r^n=0$, and the result is sharp for $n$ even. He also shown that $r_n >(1/2n)^{1/n}$ and $r_n >1-\log n /n $ for $n\geq 5$. Also he proved that the sharp radius of univalence of $f_3$ is $\sqrt 2/2$. Yamaguchi \cite{yama} has shown that $f_n$ is univalent in $|z|<1/4$ if the  analytic  function $f$ given by \eqref{fz} satisfies the inequality $\RE (f(z)/z)>0$.

Let $0\leq\alpha<1$ and denote by $R_\alpha$ the class of functions $f(z)=z+a_2z^2+\cdots$ that are regular and univalent in the unit disc and satisfy $\RE f'(z)>\alpha$. Let $f_n(z)=z+a_2z^2+\cdots+a_nz^n$. Kudelski \cite{kudel} proved the  following results. The corresponding results for $f\in R_0$ were proved by   Aksent\' ev \cite{aksent}.

\begin{theorem}Let $f\in R_\alpha$. Then  $\RE f_n{}'(z)>0$ in the disc $|z|<r_n(\alpha)$, where $r_n(\alpha)$ is the least positive root of the equation $2r^n+r-1+4\alpha r/((1-\alpha)(1+r))=0$. Also $f_n$ is univalent for $|z|<R_n(\alpha)$, where $R_n(\alpha)$ is the least positive root of $2r^n+r-1-\alpha(1-r)^2/((1-\alpha)(1+r))=0$.
\end{theorem}

\section{Close-to-convex functions}
Recall that a  function $f\in
\mathcal{A}$ satisfying the condition
\[\RE\left( \frac{f'(z)}{g'(z)}\right) > 0   \]
for some (not necessarily normalized) convex univalent function $g$,
is called \emph{close-to-convex}. In this section, some results related to close-to-convex functions are presented.


\begin{theorem}\cite{miki}  Let the analytic function $f$ be  given by \eqref{fz}. Let  $g(z)=z+b_2z^2+\cdots$ be convex.  If $\RE(f'(z)/g'(z))>0$ for $z\in\mathbb{D}$, then  $\RE(f_n'(z)/g_n'(z))>0$ for $|z|<1/4$ and  $1/4$ is the best possible constant.
\end{theorem}
The function $f$ satisfying the hypothesis of the above theorem is clearly close-to-convex.
This theorem implies that $f_n$ is also close-to-convex for $|z|<1/4$ and therefore it is a generalization  of Szeg\"o result. The result applies only to a subclass of the class of close-to-convex functions as $g $ is assumed to be normalized.
Ogawa \cite{ogawa}  proved the following two theorems.

\begin{theorem}
If  $f(z)=z+\sum_2^\infty a_\nu z^\nu$ is analytic and satisfy
\[ \RE \frac{zf'(z)}{\phi(z)}>0,\]
where $\phi(z)=z+\sum_2^\infty b_\nu z^\nu$ is starlike  univalent, then, for each $n>1$,
\[\RE \frac{(zf_n'(z))'}{\phi_n'(z)}>0\quad (|z|<{ \frac 1{8}}),\]
and the constant $1/8$ cannot be replaced by any greater one.
\end{theorem}

\begin{theorem}
Let $f(z)=z+\sum_2^\infty a_\nu z^\nu$, be analytic and satisfy
\[\RE\frac{(zf'(z))'}{\phi'(z)}>0,\] where $\phi(z)=z+\sum_2^\infty b_\nu z^\nu$ is schlicht and convex in $|z|<1$. Then, for each $n>1$,
\[ \RE \frac{zf_n'(z)}{\phi_n(z)}>0\quad (|z|<{ \frac 1{2}}).\] 
 The constant $1/2$ cannot be replaced by any greater one.
\end{theorem}

\begin{theorem}\cite{goel}
Let $f(z)=z+\sum_2^\infty a_\nu z^\nu$, be analytic and satisfy
\[\RE\frac{(zf'(z))'}{\phi'(z)}>0,\] where $\phi(z)=z+\sum_2^\infty b_\nu z^\nu$ is starlike in $|z|<1$. Then, for each $n>1$,
\[ \RE \frac{(zf_n'(z))'}{\phi_n'(z)}>0\quad (|z|<{ \frac 1{6}}).\]  The constant $1/6$ cannot be replaced by any greater one.
\end{theorem}

A domain $D$  is said to be linearly accessible if the complement of $D$ can be written as the union of half-lines. Such a domain is simply connected and therefore, if it is not the whole plane, the domain is the image of the unit disk $\mathbb{D}$ under a conformal mapping. Such conformal mappings are called linearly accessible. For  linearly accessible functions, Sheil-Small \cite{sheilsmall1} proved the following theorem.

\begin{theorem} If $f(z)=\sum_{k=1}^\infty a_k z^k$ be linearly accessible in $\mathbb{D}$, then
\[ \left|1- \frac{f_n(z)}{f(z)}\right|\leq (2n+1)|z|^n \quad (z\in \mathbb{D}). \]
\end{theorem}

\section{Partial sums of functions satisfying some coefficient inequalities}
Silverman \cite{silver-negat} considered functions  $f$ of the form $f(z)=z+\sum^\infty_{k=2}a_kz^k$ satisfying one of the inequalities \[ \sum^\infty_{k=2}(k-\alpha)|a_k|\leq(1-\alpha)\quad \text{ or  }\quad \sum^\infty_{k=2}k(k-\alpha)|a_k|\leq(1-\alpha),\] where $0\leq\alpha<1$.
These coefficient conditions are sufficient for $f$ to be starlike of order $\alpha$ and convex of order $\alpha$, respectively. If $f$ satisfies either of the inequalities above, the partial sums $f_n$ also satisfy the same inequality.

Silverman \cite{silver-negat} obtained the sharp lower bounds on $\RE\{f(z)/f_n(z)\}$, $\RE \{f_n(z)/f(z)\}$, $\RE\{f'(z)/s'_n(z)\}$, and $\RE\{s'_n(z)/f'(z)\}$  for functions $f$ satisfying either one of the inequalities above. In fact, he proved  the following theorem.

\begin{theorem} If the analytic function $f$ satisfies \[ \sum^\infty_{k=2}(k-\alpha)|a_k|\leq(1-\alpha)\] for some $0\leq \alpha<1$, then
\begin{align*}
 \RE\frac{f(z)}{f_n(z)} & \geq \frac{n}{n+1-\alpha}, \\
\RE\frac{f_n(z)}{f(z)} & \geq \frac{n+1-\alpha}{n+2-2\alpha}, \\
 \RE\frac{f'(z)}{f_n'(z)} & \geq \frac{\alpha n}{n+1-\alpha}, \\
\RE\frac{f_n'(z)}{f'(z)} & \geq \frac{n+1-\alpha}{(n+1)(2-\alpha) -\alpha} .
\end{align*} The inequalities are sharp for the function
\[ f(z)= z+\frac{1-\alpha}{n+1-\alpha}z^{n+1}.\]
\end{theorem} Silverman \cite{silver-negat} also proved a similar result for function satisfying the inequality
$\sum^\infty_{k=2}k(k-\alpha)|a_k|\leq(1-\alpha)$. These results were extended in \cite{frasin}  for classes of functions satisfying an inequality of the form $\sum c_k |a_k|\leq \delta$.

For functions
belonging to the subclass $\mathcal{S}$, it is well-known that $|a_n|\leq n$ for
$n\geq 2$. A function $f$ whose coefficients satisfy the inequality
$|a_n|\leq n$ for $n\geq 2$ are analytic in $\mathbb{D}$ (by the
usual comparison test) and hence they are members of $\mathcal{A}$.
However, they need not be univalent. For example, the function
\[f(z)=z-2z^2-3z^3-4z^4-\cdots= 2z-\frac{z}{(1-z)^2}\] satisfies the
inequality $|a_n|\leq n$ but its derivative vanishes inside
$\mathbb{D}$ and therefore the function $f$ is not univalent in
$\mathbb{D}$.  For the function $f$ satisfying the inequality
$|a_n|\leq n$, Gavrilov \cite{gavri}  showed that  the
radius of univalence of $f$ and its partial sums $f_n$  is the real root of the equation $2(1-r)^3-(1+r)=0$
while, for the functions whose coefficients satisfy $|a_n|\leq M$,
the radius of univalence is $1-\sqrt{M/(1+M)}$. Later, in 1982,
Yamashita \cite{yamashita} showed that the radius of univalence obtained by Gavrilov
is also the same as the radius of starlikeness of the corresponding
functions. He also found lower bounds for the radii of convexity for
these functions.   Kalaj,   Ponnusamy, and  Vuorinen
\cite{ponnu} have investigated related problems for harmonic
functions.  For functions of the form $f(z)=z+a_2z^2+a_3z^3+\cdots$ whose Taylor coefficients $a_n$ satisfy the conditions  $|a_2|=2b$, $0\leq b\leq 1$, and $|a_n|\leq
n $, $M$ or $M/n$ ($M>0$) for $n\geq 3$,  the sharp radii of
starlikeness and convexity of order $\alpha$, $0\leq \alpha <1$, are  obtained in
\cite{ravi}. Several other related results can also be found.

\begin{theorem}\label{th21}
Let $f\in \mathcal{A}$, $|a_2|=2b$, $0\leq b\leq 1$ and $|a_n|\leq n$ for $n\geq 3$. Then $f$
satisfies the inequality
\begin{equation*}
\left|\frac{zf'(z)}{f(z)}-1\right|\leq 1-\alpha \quad (|z|\leq r_0)
\end{equation*}
where $r_0=r_0(\alpha)$ is the real  root in $(0,1)$ of the equation
\begin{equation*}
1-\alpha+(1+\alpha)r=2\big(  1-\alpha +(2-\alpha)(1-b)r\big)(1-r)^3.
\end{equation*} The number $r_0(\alpha)$ is also the
radius of starlikeness of order $\alpha$. The number $r_0(1/2)$ is
the radius of parabolic starlikeness of the given functions.  The
results are all sharp.
\end{theorem}

\section{Partial sums of rational functions} Define $\mathcal{U}$  to be the class of all analytic functions  $f\in \mathcal{A}$ satisfying the condition
\[ \left | f'(z)\left (\frac{z}{f(z)} \right )^{2}-1\right | < 1\quad (z\in\mathbb{D}).\]It is well-known that $\mathcal{U}$ consists of only univalent functions.  In this section, we consider the partial sums of functions belonging to  $\mathcal{U}$. All the results in this section are proved by Obradovi\'c and Ponnusamy \cite{obra}.

\begin{theorem}   \label{th4}
If $f\in {\mathcal S}$    has the form\begin{equation}\label{eq-ex1}
\frac{z}{f(z)}=1+b_{1}z+b_{2}z^{2}+ \cdots
\end{equation} such that $b_k$ is real and non-negative for each
$k\geq 2$, then for each $n\geq 2$
\[\left |\frac{f_{n}(z)}{f(z)}-1\right |<|z|^{n}(n+1-n\log (1-|z|))\quad (~z\in\mathbb{D})
\]
In particular,
\[ \left |\frac{f_{n}(z)}{f(z)}-1\right |<1 \]
in the disk $|z|< r$, where $r$ is the unique positive root of the equation:
\begin{equation}\label{eq-exrevise5}
1-r^{n}(n+1-n\log (1-r))=0
\end{equation}
and, for $n\ge 3$, we also have $r\geq r_n=1-\frac{2\log n}{n}$.
\end{theorem}
The  values of $r$ corresponding to $n=2,3,4,5$
from (\ref{eq-exrevise5}) are
$r= 0.481484$, $r=0.540505$, $r=0.585302$, $r=0.620769$
respectively.

\begin{theorem} \label{U-areath1}
If $f\in \mathcal{U}$ has the form \eqref{eq-ex1}, then
\begin{equation}\label{09-eq5}
\sum_{n=2}^{\infty}(n-1)^{2}|b_{n}|^{2}\leq 1.
\end{equation}
In particular, we have
$|b_{1}|\leq 2$ and $|b_{n}|\leq\frac{1}{n-1}$ for $n\geq2$. The results
are sharp.
\end{theorem}

\begin{theorem}\label{th4-e}
Suppose that  $f\in {\mathcal U}$ and $f_{n}(z)$ is its partial sum. Then for each $n\geq 2$
$$\left |\frac{f_{n}(z)}{f(z)}-1\right |<|z|^{n}(n+1)\left(
 1+\frac{\pi}{\sqrt{6}}\frac{|z|}{1-|z|} \right)\quad (~z\in\mathbb{D}).
$$
\end{theorem}

\begin{proof}
Let $f(z)=z+a_{2}z^{2}+a_{3}z^{3}+\cdots $ so that
$$f_n(z)=z+a_{2}z^{2}+a_{3}z^{3}+\cdots +a_{n}z^{n}
$$ is its $n$-th partial sum. Also, let
$$ \frac{z}{f(z)}=1+b_{1}z+b_{2}z^{2}+ \cdots .
$$
Then
$$\frac{f_{n}(z)}{f(z)}= 1+c_{n}z^{n}+c_{n+1}z^{n+1}+\cdots
$$
where $c_{n}=-a_{n+1}$ and
$$c_{m}=b_{m-n+1}a_{n}+b_{m-n+2}a_{n-1}+ \cdots + b_{m}a_{1},
$$
for $m=n+1, n+2, \ldots .$  By de Branges theorem, $|a_{n}|\leq n$ for
all $n\geq 2$, and therefore, we obtain that
$$|c_{n}|=|-a_{n+1}|\leq n+1
$$
and that for  $m\geq  n+1$ (by the Cauchy-Schwarz inequality)
\[ |c_{m}|^2 \leq \left (\sum_{k=1}^{n}\frac{(n+1-k)^{2}}{(m-(n+1-k))^{2}}\right ) \times \]\[
\left (\sum_{k=1}^{n}(m-n+k-1)^{2}\, |b_{m-n+k}|^{2}\right ):= A B
~ \mbox{(say)}.\]
From Theorem \ref{U-areath1}, we deduce that $B\leq1$, while for  $m\geq  n+1$ we have
\begin{align*}
A& =\sum_{k=1}^{n}\frac{(n+1-k)^{2}}{(m-(n+1-k))^{2}}\\
& \leq   \sum_{k=1}^{n}\frac{(n+1-k)^{2}}{k^{2}}\\
& =  (n+1)^{2}\sum_{k=1}^{n}\frac{1}{k^{2}}-2(n+1)\sum_{k=1}^{n}\frac{1}{k}+\sum_{k=1}^{n}1.
\end{align*} in view of the inequalities
$$\sum_{k=1}^{n}\frac{1}{k}>\log (n+1), ~\mbox{ and }~ \log (n+1) >1 \mbox{ for $n\geq 3$},
$$
it follows easily that
$$A<\frac{\pi^{2}}{6}(n+1)^{2}-2(n+1)\log(n+1)+n<\frac{\pi^{2}}{6}(n+1)^{2} -(n+2),
$$
which, in particular, implies that
$$|c_{m}|<\frac{\pi}{\sqrt{6}}(n+1) ~\mbox{ for $m\geq n+1$ and $n\geq 3$}.
$$
This inequality, together with the fact that $|c_{n}|=|a_{n+1}|\leq n+1$, gives that
\begin{align*}
& \left |\frac{f_{n}(z)}{f(z)}-1\right |\\
& \leq   |c_{n}|\,|z|^{n}+|c_{n+1}|\, |z|^{n+1}+ \cdots \\
&\leq    (n+1)|z|^{n}+\frac{\pi}{\sqrt{6}}(n+1) \left (|z|^{n+1}+|z|^{n+2}+\cdots \right ) \\
&=  (n+1)|z|^{n}\left (1+ \frac{\pi}{\sqrt{6}}\frac{|z|}{1-|z|}\right )
\end{align*}
for $n\geq 3$. The proof is complete.
\end{proof}

As a corollary, the following result holds.

\begin{corollary}\label{coro-new1}
Suppose that $f\in \mathcal{U}$. Then for $n\geq 3$ one has
$$\left |\frac{f_{n}(z)}{f(z)}-1\right |<\frac{1}{2} ~\mbox{ for }~ |z|<r_n :=1-\frac{2\log n}{n}
$$
or equivalently
$$\left |\frac{f(z)}{f_{n}(z)} -\frac{4}{3}\right |<\frac{2}{3} ~\mbox{ for }~ |z|<r_n  .
$$
\end{corollary}

In particular, Corollary \ref{coro-new1} shows that for $f\in \mathcal{U}$, we have
$${\rm Re}\,\frac{f_{n}(z)}{f(z)}>\frac{1}{2} ~\mbox{ for
$|z|<r_n$ and $n\ge 3$}
$$
and
$${\rm Re}\,\frac{f(z)}{f_{n}(z)} >\frac{2}{3} ~\mbox{ for $|z|<r_n$ and $n\ge 3$}.
$$

When the second Taylor coefficient of the function $f$ vanish, the following results hold.

\begin{theorem}\label{th2}
If $f(z)=z+\sum_{k=3}^{\infty}a_{k}z^{k}$ $($i.e. $a_2=0)$ belongs to the class $\mathcal{U}$,
then the $n$-th partial sum $f_n$  is in the class $\mathcal{U}$
in the disk $|z|< r$, where $r$ is the unique positive root of the equation
$$(1-r)^3(1+r)^2-r^n(1+r^2)^2[5+r+n(1-r^2)]=0.
$$
In particular, for $n\ge 5$, we have \[ r\geq r_n=1-\frac{3\log n-\log(\log n)}{n}.\]
\end{theorem}

For $n=3,4,5$, one has
$$r=0.361697, ~r=0.423274, ~r=0.470298,
$$
respectively.

\begin{theorem}\label{th3}
Let $f(z)=z+\sum_{k=3}^{\infty}a_{k}z^{k}$ $($i.e. $a_2=0)$ belong to the class $\mathcal{U}$. Then for each
integer $n\geq 2$, we have
$${\rm Re}\, \left (\frac{f(z)}{f_n(z)}\right )>\frac{1}{2}
$$
in the disk $|z|< \sqrt{\sqrt{5}-2}.$
\end{theorem}

\section{Generalized Partial Sum} By making use of the fact that the convolution of starlike with a convex function is again a starlike function, Silverman \cite{silver88} has proved the following result.
\begin{theorem}If $f(z)=z+\sum_{k=2}^\infty a_k z^k$ is convex, then $F_k(z)=z+\sum_{j=1}^\infty a_{jk+1}z^{jk+1}$, $(k=2,3,\dotsc)$, is starlike in $|z|<(1/(k-1))^{1/k}$. The bound is sharp for every $k$.
\end{theorem}
The proof follows from the following inequality satisfied by $G_k(z)=z/(1-z^k)$:
\[\RE\frac{zG_k'(z)}{G_k(z)} \geq \frac{1-(k-2)r^k-(k-1)r^{2k}}{|1-z^k|^2}\quad (|z|=r<1).\]
Since $(1/(k-1))^{1/k}$  attains its minimum when $k=5$, it follows that, for a convex function $f$,  the $F_k$ is starlike in $|z|<(1/4)^{1/5}$. Since the radius of convexity of $G_2$
is $\sqrt{2}-1$, it follows that $F_2$ is convex in $|z|<\sqrt{2}-1$ whenever $f$ is convex.

To an arbitrary increasing sequence (finite or not) $\{n_k\}^\infty_{k=2}$ of integers with $n_k\geq k$ and a function $f(z)=z+\sum_{k=2}^\infty a_k z^k$, the function $\tilde{f}(z)=
z+\sum_{k=2}^\infty a_{n_k} z^{n_k}=f(z)*(z+\sum_{k=2}^\infty  z^{n_k})$ is called the generalized partial sum of the function $f$.  For the generalized partial sum of convex mappings, Fournier and Silverman \cite{four} proved the following results.

\begin{theorem}\label{four-silver}
If $f$ is convex, then the generalized partial sum $\tilde{f}$ of the function $f$ is
\begin{enumerate}
  \item convex univalent in $|z|<c$ where $c\;(\approx 0.20936)$ is the unique root in $(0,1)$ of the equation \[ x(1+x^2)/(1-x^2)^3=1/4.\]
  \item starlike  univalent in $|z|<b$ where $b\;(\approx 0.3715)$ is the unique root in $(0,1)$ of the equation \[x/(1-x^2)^2=1/2.\]
\end{enumerate}
The function $z+\sum^\infty_{k=1} z^{2k}=z+z^2/(1-z^2)$ associated with the convex function $z+\sum^\infty_{k=2} z^k=z/(1-z)$ is extremal for the radii of convexity and starlikeness.
\end{theorem}
These results are proved by using the information about neighborhoods of convex functions.
They \cite{four} also proved  that, for a   starlike function $f$, the generalized partial sum  $\tilde{f}$ is starlike in $|z|<c$ where $c$ is  as above or in other words,
\begin{equation}\label{fs-imp}  f\in H \Rightarrow \tilde{f}(cz)/c\in H\end{equation}
where $H$ is the class of starlike univalent functions. The above implication in \eqref{fs-imp} is also valid for the classes of convex univalent functions and close-to-convex functions and the class $M$ consisting of functions $f$ for which $(f*g)(z)/z\neq 0$ for all starlike univalent functions $g\in \mathcal{S}^*$. They \cite{four2} later showed that the implication in \eqref{fs-imp} is also valid for the class $\mathcal{S}$ of univalent functions by proving the following theorem.

\begin{theorem}If $f\in \mathcal{S}$, then the generalized partial sum $\tilde{f}$ of the function $f$ satisfies $\RE \tilde{f}'(cz)>0$ for all $z\in \mathbb{D}$, where $c$ is as in Theorem~\ref{four-silver}. The function $f(z)=z/(1-z)^2$ and $\{n_k\}^\infty_{k=2}=\{2k-2\}^\infty_{k=2}$ show that the result is sharp.
\end{theorem}

They \cite{four2}  have  also proved that if $f$ is analytic and  ${\rm Re}\{f(z)/z\}>\frac12$, then \[ |z\tilde f''(z)|\leq{\rm Re}\,\tilde f'(z) \quad (|z|<c)\] for any choice of $\{n_k\}^\infty_{k=2}$.

For the class $\mathcal{R}$  of functions $f$ in $\mathcal{A}$ for which ${\rm Re}(f'(z)+zf''(z))>0$, $z\in\mathbb{D}$,  Silverman \cite{silverR}  proved the following result and some related results can be found in \cite{silver96}.

\begin{theorem}
Let $r_0$ denote the positive root of the equation $r+\log(1-r^2)=0$. If $f\in \mathcal{R}$, then ${\rm Re}\, \tilde{f}{}'(z)\geq 0$ for $|z|\leq r_0\approx 0.71455$. The result is sharp, with extremal function $\tilde f(z)=z+2\sum^\infty_{n=1}z^{2n}/(2n)^2$.
\end{theorem}

For functions $f\in \mathcal{R}$, it is also known \cite{singhsingh} that
the $n$th partial sum $f_n$ of $f$ satisfies $\RE f_n'(z)>0$ and hence $f_n$ is univalent.  Also $\RE (f_n(z)/z)>1/3$.

\section*{Acknowledgement} The author is thankful to Sumit Nagpal for carefully reading  this manuscript.

\end{document}